\documentclass[11pt,openany,leqno]{article}
\usepackage{amsmath,amsthm,amsfonts,amssymb,amscd,url}
\usepackage{graphics}

\input{xy} \xyoption{all}
\usepackage[latin1]{inputenc}

\usepackage{enumerate}
\usepackage{hyperref}
\usepackage{epsfig} 
\input{xy} \xyoption{all}

\begin{document}
\def\Diff{\text{Diff}}
\def\Max{\text{max}}
\def\R{\mathbb R}
\def\N{\mathbb N}
\def\Z{\mathbb Z}
\def\Q{\mathbb Q}
\def\a{{\underline a}}
\def\b{{\underline b}}
\def\c{{\underline c}}
\def\Log{\text{log}}
\def\loc{\text{loc}}
\def\inta{\text{int }}
\def\det{\text{det}}
\def\exp{\text{exp}}
\def\Re{\text{Re}}
\def\lip{\text{Lip}}
\def\leb{\text{Leb}}
\def\dom{\text{Dom}}
\def\diam{\text{diam}\:}
\def\supp{\text{supp}\:}
\newcommand{\ovfork}{{\overline{\pitchfork}}}
\newcommand{\ovforki}{{\overline{\pitchfork}_{I}}}
\newcommand{\Tfork}{{\cap\!\!\!\!^\mathrm{T}}}
\newcommand{\whforki}{{\widehat{\pitchfork}_{I}}}
\newcommand{\marginal}[1]{\marginpar{{\scriptsize {#1}}}}

\theoremstyle{plain}
\newtheorem{theo}{\bf Theorem}[section]
\newtheorem{lemm}[theo]{\bf Lemma}
\newtheorem{conj}[theo]{\bf Conjecture}
\newtheorem{ques}[theo]{\bf Question}
\newtheorem{prb}[theo]{\bf Problem}
\newtheorem{sublemm}[theo]{\bf Sublemma}
\newtheorem{IH}[theo]{\bf Extra induction hypothesis}
\newtheorem{prop}[theo]{\bf Proposition}
\newtheorem{coro}[theo]{\bf Corollary}
\newtheorem{Property}[theo]{\bf Property}
\newtheorem{Claim}[theo]{\bf Claim}
\theoremstyle{remark}
\newtheorem{rema}[theo]{\bf Remark}
\newtheorem{important rema}[theo]{\bf Important remark}
\newtheorem{remas}[theo]{\bf Remarks}
\newtheorem{fact}[theo]{\bf Fact}

\newtheorem{exem}[theo]{\bf Example}
\newtheorem{Examples}[theo]{\bf Examples}
\newtheorem{defi}[theo]{\bf Definition}


\title{Nested Cantor sets}

\author{Pierre Berger~\footnote{berger@math.univ-paris13.fr, LAGA
    Université Paris 13, partially financed by the Balzan prize of J. Palis}{ } and Carlos Gustavo Moreira~\footnote{gugu@impa.br, Instituto de Matem\'atica Pura e Aplicada; partially financed by the Balzan prize of J. Palis and by CNPq}{} }

\date{\today}

\maketitle
\begin{abstract} 
We give sufficient conditions for two Cantor sets of the line  to be nested for a positive set of translation parameters.
This problem occurs in diophantine approximations. It also occurs as a toy  model of the parameter selection for non-uniformly hyperbolic attractors of the plane.   For natural Cantors sets, 
we show that this condition is optimal.  
\end{abstract}
\section*{Introduction}\label{intro}

One dimensional Cantor sets of positive measure appear frequently in dynamical systems and diophantine geometry. 

 For instance, if  a quadratic map $f(x)=x^2+a$  has an absolutely continuous invariant measure (acim) with Lyapunov exponent $\lambda>0$, then for  $C>0 $ small enough, $\lambda' \in (0,\lambda)$, the set $\tilde K$ of points $x$ such that:

\begin{equation}\tag{$\star$} \label{Explyap}
 \|D_xf^n\| \ge  {C\lambda'^n},\quad \forall n\ge 0.
\end{equation}
is a Cantor set of positive Lebesgue measure, called \emph{Pesin set}.  Collet-Eckmann property is that the critical value belongs to such a $\tilde K$ (for a certain  $C,\lambda'$). Actually the set of Collet-Eckmann parameters $a$ is of positive measure \cite{BC1, Yoc}. A rough idea of the proof is the following: the critical value $a$ ``moves faster'' than the set $\tilde K$.

In higher dimension, for instance in the study of the Hénon map \cite{berhen}, a non-uniformly hyperbolic attractor appears when a  Cantor set $K$ is included in a Cantor set similar to  $\tilde K$. A main issue is to find such non-uniformly hyperbolic attractors with a Hausdorff dimension not too close to one,  for a positive set of parameters of generic family of maps of surface.
 As a toy model for the parameter selection, we can assume that $K$ and $\tilde K$ do not depend on the parameter.  This leads to the following question:

\begin{prb}[\cite{berhen}]\label{question}
Under which hypothesis on a  Cantor set $K$ and a Lebesgue positive Cantor set $\tilde K$, the set of parameters $t\in \R$ such that $K+t$ is included in $\tilde K$ has a positive Lebesgue measure?
\end{prb}
 
We will give a sufficient condition for this problem, based on the following condition on $\tilde K$.

\begin{defi} For $0<p< 1$, a compact set $\tilde K = [a,a+\diam\, \tilde K]\setminus \sqcup_n (a_n,b_n)$, where the intervals $(a_n,b_n)$ are the connected components of $[a,a+\diam\, \tilde K]\setminus \tilde K$ satisfies condition $(\mathcal C_p)$ if:
\begin{equation}\tag{$\mathcal C_p$}\sum_n(b_n-a_n)^p<\infty.\end{equation} 
We define
\[P(\tilde K):= \inf\{p>0:\; \tilde K \text{ satisfies condition }(\mathcal C_p)\}\]
\end{defi}

Let us formulate a few remarks on condition ($\mathcal C_p$).
\begin{rema}\label{exprediff} If the complement of $\tilde K$ is expressed as an union of intervals which are not necessarily connected components of $[a,a+\diam\, \tilde K]\setminus \tilde K$:
\[\tilde K = [a, a+\diam\, \tilde K]\setminus \cup_n (a_n,b_n),\]
 such that $\sum_n(b_n-a_n)^p<\infty$, then $\tilde K$ satisfies condition ($\mathcal C_p$), with $0<p<1$. Indeed, for every positive numbers  $(l_i)_i$ , it holds  $\sum_i l_i^p\ge (\sum_i l_i)^p$, hence the sum  $\sum_n(b_n-a_n)^p$ is smaller if we group the intervals in the same components.
\end{rema}
The inequality $\sum_i l_i^p\ge (\sum_i l_i)^p$ also implies the following
\begin{rema} 
If $\tilde K_1$ and $\tilde K_2$ are two subsets of $[0,1]$ satisfying ($\mathcal C_p$), then 
 $\tilde K_1\cap \tilde K_2$ satisfies ($\mathcal C_p$). 
\end{rema}

 A consequence of the main result is the following:
\begin{prop}\label{coro1}Let $K$ be a regular Cantor set of Hausdorff dimension $d$ of $\mathbb R$ and let $\tilde K$ be a Cantor set of positive Lebesgue measure satisfying condition ($\mathcal C_p$). If $d+p<1$ then for $\lambda>0$ sufficiently small, the following set has positive Lebesgue measure:
\[\{t\in \mathbb R:\; t+\lambda K\subset   \tilde K\}.\]
\end{prop}  

Main Theorem \ref{theo1} gives a sufficient condition to answer to Problem \ref{question}. It implies the above Proposition. It implies also the following results on the in diophantine approximation geometry. 

Given a real number $\alpha$ and a real number $d\ge 2$, we say that $\alpha$ is {\it diophantine of order $d$} if there is a real constant $c>0$ such that $|\alpha-\frac{p}{q}|>\frac{c}{q^d}$, for all integers $p, q$ with $q>0$. Let $D_d$ be the set of such numbers $\alpha$. We say that $\alpha$ is {\it diophantine} if $\alpha$ is diophantine of order $d$ for some $d\ge 2$. Let $D$ be the set of the diophantine numbers. Many metric results about diophantine numbers are well-known. For instance, if $d>2$ then $D_d$ has full Lebesgue measure, and $\mathbb R\setminus D_d$ has Hausdorff dimension $2/d$ (see for instance \cite{Ja2}). As a consequence, the set of Liouville numbers $L=\mathbb R\setminus(\mathbb Q \cup D)$ has zero Lebesgue measure (but is a residual set in Baire's sense). On the other hand, $D_2$ has zero Lebesgue measure but Hausdorff dimension $1$ (see \cite{Ja1} or \cite{CF}). And, by the classical Dirichlet's theorem on diophantine approximations, for $d<2$ the set $D_d$ is empty. 
 Here is an application of the main result. 

\begin{theo}\label{dio} If $K$ is a regular Cantor set of Hausdorff dimension $s$ with $s<1-\frac2{d}$, then, for almost every $t\in \mathbb R$, $K+t\subset D_d\subset D$.\end{theo}
We can apply this result to the set of {\it badly approximable} real numbers, that is those which are diophantine of order $2$. The set $B=D_2$ of badly approximable real numbers is the countable union for all integers $j$ and all positive integers $k$ of the regular Cantor sets $F_k+\{j\}$, where $F_k\subset [0,1]$ is the regular Cantor set (defined by a suitable restriction of the Gauss map $g:(0,1)\to [0,1), g(x)=1/x-\lfloor 1/x\rfloor$) given by  $F_k=\{\alpha=[0;a_1,a_2,...];1\le a_j\le k, \forall j\ge 1\}$ (where $[0;a_1,a_2,...]$ is the usual representation of $\alpha$ by continued fractions; see \cite{CF}).
\begin{coro}For almost every $t\in \mathbb R$, $B+t\subset D$.\end{coro}

We define in section \ref{Notations and definitions} the different notions used by Proposition \ref{coro1} and Theorem \ref{theo1}. 
In section \ref{Main result and its applications}, we state the main theorem and prove Proposition \ref{coro1} and Theorem \ref{dio}. Then we apply this Proposition to many examples which are toy models for the parameter selection of non-uniformly hyperbolic attractors.
In section \ref{prooftheo1}, Theorem \ref{theo1} is proven. In section \ref{Discussion}, we discuss on the optimality of the result, thanks to examples and counterexamples. The examples optimality are rather natural whereas the counterexamples are rather artificial.
In section \ref{Explanation of the toy  model of the parameter selection for non-uniformly hyperbolic attractors}, we will apply these results to example from non uniformly hyperbolic dynamics. They serve as a toy model for parameter selection of non-uniformly hyperbolic attractor of surface diffeomorphisms. These toy models are then explain, in relation to a conjecture of Hénon.

\section{Notations and definitions}\label{Notations and definitions}
\subsection{Operations on sets}
Given two subsets $E,F\subset \R$, the arithmetic sums $E+F$ and $E-F$ stand for the following subsets of $\R$:
\[ E+F=\{e+f:\; e\in E,\; f\in F\},\quad E-F=\{e-f:\; e\in E,\; f\in F\}.\]

We denote also by $-E$ the set $\{0\}-E$.

The operations $+$ and $-$ are associative.  Thus, for any three subsets $E,F,G\subset \R$, we can denote by $E+F+G$ the set:
\[E+F+G=E+(F+G)=(E+F)+G.\]

For $U\in \mathbb R$, let $U\cdot E:=\{U\cdot e:\; e\in E\}$. However $2\cdot E\subsetneq E+E$ if $E$ is not trivial.

\subsection{Regular Cantor sets}  
\begin{defi}[Regular Cantor set]
We recall that $K$ is a {\it $C^k$-regular Cantor set}, $k\ge 1$, if:

\begin{itemize}

\item[i)] there are disjoint compact intervals $I_1,I_2,\dots,I_r$ such that $K\subset I_1 \cup \cdots\cup I_r$ and the boundary of each $I_j$ is contained in $K$;

\item[ii)] there is a $C^k$ expanding map $\psi$ defined in a neighbourhood of $I_1\cup I_2\cup\cdots\cup I_r$ such that $\psi(I_j)$ is the convex hull of a finite union of some intervals $I_s$ satisfying:

\begin{itemize}

\item[ii.1)] for each $j$, $1\le j\le r$ and $n$ sufficiently big, $\psi^n(K\cap I_j)=K$;

\item[ii.2)] $K=\bigcap\limits_{n\in\mathbb N} \psi^{-n}(I_1\cup I_2\cup\cdots\cup I_r)$.

\end{itemize}
\end{itemize}
\end{defi}
There are two basic  concepts of dimension for Cantor sets:
\begin{defi}[Box dimension and box fuzzy measure]
 The \emph{box dimension} $d$ of a compact set $K$ is the infimum of $s\ge 0$ such that the following is zero:
 \[\limsup_{\epsilon\rightarrow 0}\; \inf\{N\epsilon^{s} :\;(U_i)_{i=1}^N\;\text{is a finite open cover of }K\text{ by}\text{ intervals of diameter }\epsilon\}.\] 
We denote by $C_K$ the following supremum, with $d$ the box dimension:
\[ C_K:= \sup_{0< \epsilon\le \diam \, K}\; \inf\{N\epsilon^{d} :\;(U_i)_{i=1}^N\;\text{is a finite open cover of }K\text{ by intervals of diameter }\epsilon\}.\] 
The map $K\mapsto C_K$ is a fuzzy measure, but not a measure. We call it the \emph{$d$-box fuzzy measure}. 
\end{defi}
\begin{rema}For every compact set $K$ of box dimension $d$  and $\lambda>0$, it holds:
\[ \diam(\lambda K)= \lambda \diam K\quad \text{and} \quad C_{\lambda K} = \lambda^d C_K.\] 
\end{rema}
The box dimension is at least equal to the Hausdorff dimension:
\begin{defi}[Hausdorff dimension]
 The \emph{Hausdorff dimension} $d$ of a compact set $K$ is the infimum of $s$ such that the following is zero:
 \[\limsup_{\epsilon\rightarrow 0}\; \inf\{\sum_{i} \diam(U_i)^s :\;(U_i)_{i}\;\text{is an open cover of }K\text{ by intervals of diameter}\le \epsilon\}.\]
  The value of the above limit at $s=d$ is denoted by $m_d(K)$. The map $K\mapsto m_d(K)$ is a measure, called the Hausdorff $d$-dimensional measure. 
\end{defi}
In the case of regular Cantor set, these two definitions coincide from the following.
\begin{theo}[Palis-Takens  Prop. 3 p. 72 \cite{PT}]\label{PT}
If $K$ is a regular Cantor set of class $C^{1+\alpha}$, the Hausdorff dimension $d$ of $K$ is equal to the box dimension of $K$. Moreover, the box fuzzy measure $C_K$ is finite. 
\end{theo}

\section{Main result and its applications}\label{Main result and its applications}
Here is the main result of this article. Its proof will be done in \textsection \ref{prooftheo1}.
\begin{theo}\label{theo1}
Let $\tilde K= [0,\diam\, \tilde K]\setminus \sqcup_n (a_n,b_n)$ be a Cantor set satisfying condition ($\mathcal C_p$). Put $l_n:= b_n-a_n$. Let $K$ be a regular Cantor set of dimension $d< 1-p$ and box fuzzy measure $C_K$ satisfying:
\[\sum_{\{n: \; l_n > \diam\, K\}} (\diam\, K +l_n) + 2 C_K \sum_{\{n:\; l_n \le \diam\, K\}}  (l_n)^{1-d}< \diam(\tilde K)-\diam(K),\]
then the set of parameters $t\in \mathbb R$ such that $t+K$ is included in $\tilde K$ has  Lebesgue positive measure. More precisely, its Lebesgue measure is at least
\[\diam(\tilde K)-\diam(K)-\big(\sum_{\{n: \; l_n > \diam\, K\}} (\diam\, K +l_n) + 2 C_K \sum_{\{n:\; l_n \le \diam\, K\}}  (l_n)^{1-d}\big).\]
\end{theo}

\begin{rema}  If $K$ is a set of positive measure and  $t+K$ is included in  a Cantor set $\tilde K$ for a Lebesgue positive set $T$ of parameters $t$, then 
$K+T\subset \tilde K$ has non-empty interior (which is impossible since $\tilde K$ is a Cantor set). 

Indeed, $K+T=K-(-T)=\{s| K \cap (s-T) \ne \emptyset\}$. Since $K$ and $-T$ have positive measure, there is $\epsilon>0$ and intervals $[a,a+\epsilon]$ and $[b,b+\epsilon]$ such that $\leb([a,a+\epsilon]\cap K)>0.99 \epsilon$ and $\leb([b,b+\epsilon]\cap (-T))>0.99  \epsilon$. 
 Hence $\leb([a,a+\epsilon]\cap (a-b-T))>0.99 \epsilon$, and so $a-b$ is in the interior of $K+T$. 
\end{rema}

\begin{proof}[Proof of  Proposition  \ref{coro1}] Without lost of generality we can suppose that the diameter of $K$ is $1$, and hence for $\lambda >0$:  
\[ \diam(\lambda K)= \lambda. \]
 We recall that for every regular Cantor set $K$ of dimension $d$  and $\lambda>0$, it holds $  C_{\lambda K} = \lambda^d C_K$.
 
By Theorem \ref{theo1},  Proposition \ref{coro1} holds if  for $\lambda$ small:
\begin{equation}\label{estim1}\sum_{\{n: \; l_n > \lambda\}} (\lambda +l_n) + 2 \lambda^d C_K \sum_{\{n:\; l_n \le \lambda\}}  (l_n)^{1-d}< \diam(\tilde K)-\lambda,\end{equation}
where $(l_n)_n$ are the diameter of the holes of $\tilde K$. As condition ($\mathcal C_{1-d}$) is satisfied by $\tilde K$,  inequality (\ref{estim1}) holds if there exists $\eta>0$ such that for $\lambda $ small:
\begin{equation}\label{estim2}\sum_{\{n: \; l_n > \lambda\}} (\lambda +l_n)< \diam(\tilde K)-\eta.\end{equation}
We remark that for every $N\ge 0$:
\[\sum_{\{n: \; l_n > \lambda\}} (\lambda +l_n) = 
\sum_{\{n\le N: \; l_n > \lambda\}} (\lambda +l_n)+\sum_{\{n> N: \; l_n > \lambda\}} (\lambda +l_n)
\le 
\sum_{\{n\le N: \; l_n > \lambda\}} (\lambda +l_n)+\sum_{n> N} 2l_n.\]
For $N$ large the second sum is small, and then by taking $\lambda $ small (in function of $N$), the first sum is close to 
$\sum_{n} l_n=\leb([0,\diam\, \tilde K]\setminus \tilde K)$. Consequently  inequality (\ref{estim2}) holds for $\lambda$ small if 
\[\leb([0,\diam\, \tilde K]\setminus \tilde K)< \diam(\tilde K)\]
which is indeed the case since $\tilde K$ has positive Lebesgue measure.
\end{proof}

Let us apply Theorem \ref{theo1} and  Proposition \ref{coro1} to a simple example and to diophantine geometry.

\begin{exem} Let $\tilde K(s)$ be the Cantor set $\cap_{n\ge N} \tilde K_n$ where $\tilde K_n$ is an union of $2^n$ intervals constructed by induction as follows:
\begin{itemize} \item $\tilde K_0:= [-1,1]$,
\item for $n>0$, $\tilde K_{n}$ is obtained by removing from each component of $\tilde K_{n-1}$ an open interval of length $ \frac{2^{-n/s}}{1-2^{1-1/s}}$, for instance in the middle of the component. 
\end{itemize}
The set $[-1,1]\setminus \tilde K(s)$ has $2^{n}$ components of length $2^{-n/s}$. Hence this Cantor set satisfies Condition ($\mathcal C_p$) for $p> s$ (and not satisfied for $p \le s$), that is, $P(K(s))=s$. From the definition of $n_0$, one checks that $\tilde K(s)$ has Lebesgue measure equals to $1$.

Hence by  Proposition \ref{coro1} the Cantor set $\tilde K(s)$ contains a positive set of translations of a dyadic Cantor set $K_d$ of dimension $d<1-s$. 
\end{exem}
%

\begin{proof}[Proof of Theorem \ref{dio}]
Let $[a,b]$ be the support interval of $K$. Let $M$ be a large positive integer. We will estimate the measure of the set of values of $t\in [-M,M]$ for which $K+t \subset D_d$. Take $q_0$ a large positive integer. The set $D_d$ contains the set
\[ X=[a-M,b+M]\setminus\bigcup_{q\ge q_0}\bigcup_{p\in \mathbb Z, p/q \in [a-M,b+M]}\big(\frac pq-\frac1{q^d},\frac pq+\frac1{q^d}\big).\]
For $q_0$ large, $q\ge q_0$ and $p\in \mathbb Z$, the size of the interval $(\frac pq-\frac1{q^d},\frac pq+\frac1{q^d})$ is $\frac2{q^d}\le\frac2{q_0^d}<b-a=\diam(K)$. So, by Theorem \ref{theo1}, the Lebesgue measure of the set of $t\in [-M,M]$ such that $K+t\subset X \subset D_d$ is at least
\[(2M+b-a)-(b-a)-2C_K\sum_{q\ge q_0}\sum_{p\in \mathbb Z, p/q \in [a-M,b+M]}(\frac2{q^d})^{1-s}\ge 2M-2C_K\sum_{q\ge q_0}4Mq(\frac{2}{q^d})^{1-s}>\]
\[>2M-8MC_K\sum_{q\ge q_0}\frac1{q^{(1-s)d-1}}>2M-\frac{16MC_K}{((1-s)d-2)q_0^{(1-s)d-2}}.\]
Since $s<1-\frac2{d}$, $(1-s)d>2$, and $(1-s)d-2>0$, so this Lebesgue measure tends to $2M$ when $q_0$ grows. This implies that the set of $t\in \mathbb R$ such that $K+t\subset D_d\subset D$ has full Lebesgue measure, and we are done.
\end{proof}

\section{Proof of Theorem \ref{theo1}}\label{prooftheo1}
We recall that $\tilde K$ is included in $[0, \diam\, \tilde K]$ with $[0,\diam\, \tilde K]\setminus \tilde K= \sqcup_n (a_n, b_n)$.
 
We can suppose that the regular Cantor set $K$ of dimension $d$ is included in $[0,\diam\, K]$.
 Observe that if $t$ is  in $[0, \diam\, \tilde K-\diam\, K]$, then $t+K$ is included in $[0,\diam\, \tilde K]$.

Let $\mathbb X$ be the set of parameters $t$ such that $t+K\subset \tilde K$.

Let us consider $t\in [0, \diam\, \tilde K-\diam\, K]\setminus \mathbb X$, in other words $K+t$ is not included in $\tilde K$. 

This means that $K+t$ intersects the interval $(a_n,b_n)$ for a certain $n$. That is to say: $t\in (a_n,b_n)-K$.
 This implies:
\begin{fact} \label{lefait}
$[0,\diam\, \tilde K-\diam\, K]\setminus \mathbb X= \bigcup_{n}((a_n,b_n)-K)$.  
\end{fact}
A direct consequence is the following nice inequality:
\begin{coro}\label{evaluationdesmesures} $\leb([0,\diam\, \tilde K-\diam\, K]\setminus \mathbb X)\le \sum_{n}\leb((a_n,b_n)-K)$.\end{coro}

The measure of  $(a_n,b_n)-K$ is estimated by the following Proposition shown below.

\begin{prop}\label{propo1}\label{relationdimmesinter} For every interval $I$ of $\R$, we have:
\[\left\{\begin{array}{cl} \leb\left(I-K\right)\le 2 C_K ({\leb \, I})^{1-d}& \text{if } \diam\, I\le \diam\,K,\\
\leb\left(I-K\right) \le \leb \, I +\diam\, K& \text{if } \diam\, I\ge \diam\, K,\end{array}\right.\]
with $C_K$ the box fuzzy measure of dimension $d$.
\end{prop}
From  Proposition \ref{relationdimmesinter}, with $l_n= b_n-a_n$:
\[\leb([0,\diam\, \tilde K-\diam\, K]\setminus \mathbb X)\le \sum_{n}\leb((a_n,b_n)-K)\]\[\le  
\sum_{\{n:\; l_n > \diam\, K\}} (\diam\, K+ l_n)+\sum_{\{n:\; l_n \le \diam\, K\}} 2 C_K {l_n}^{1-d}
,\]
which implies Theorem \ref{theo1}.

\begin{proof}[Proof of Proposition \ref{propo1}]
To see the case $\diam\, I\ge \diam\, K$ consider that $K\subset[0, \diam\, K]$ and that $I=[-\diam \, I,0]$ (this can be supposed via translation). Then $K-I= [0, \diam\, K+\diam\, I]$.

Let us study the case $\diam(I)\le \diam(K)$.
We can assume $I=[-\epsilon/2,\epsilon/2]$ with $\epsilon=\leb \, I$. 

First observe that
 \begin{equation}\label{K-(a_n,b_n)}[-\epsilon/2,\epsilon/2]-K=-(K-[-\epsilon/2,\epsilon/2])=-(K+[-\epsilon/2,\epsilon/2]).
\end{equation}
Note that $K+[-\epsilon/2,\epsilon/2]$ is the closure of the $\frac \epsilon 2$-neighborhood of $K$.
 
By definition of $C_K$, the Cantor set $K$ is included in the union of $C_K \epsilon^{-d}$ intervals of length $\epsilon$. The $\epsilon/2$-neighborhood of this union is contained in  the union $U$ of $C_K \epsilon^{-d}$ intervals of length $2\epsilon$, the Lebesgue measure of which is 
\[\leb(U)\le  2 C_K \epsilon^{1-d}= 2C_K (\leb \, I)^{1-d}.\]
As $U$ contains $K+ [-\epsilon/2,\epsilon/2]$, this proves the Proposition.
\end{proof}

\section{Discussion on the optimality of Condition ($\mathcal C_p$)}\label{Discussion}
\subsection{Example for which  Proposition \ref{coro1} is trivially non optimal}
Condition ($\mathcal C_p$) is not optimal for Proposition \ref{coro1}. This follows from the following observations. 

Consider  two Cantor sets $\tilde K_1$ and $\tilde K_2$ of positive Lebesgue measure and satisfying respectively $\mathcal C_{p_1}$ and $\mathcal C_{p_2}$. 

The disjoint union $\tilde K$  of two Cantor sets $\tilde K_1$ and $\tilde K_2$ is of positive Lebesgue measure and satisfies, in general,  $\mathcal C_{p}$ only for $p\ge \max (p_1,p_2)$, whereas Proposition \ref{coro1} applied to $\tilde K_1$ and $\tilde K_2$ implies that it contains a Lebesgue positive set of translations of a Cantor set with Hausdorff dimension $1- \min(p_1,p_2)$.

In the seek of an optimal result,  the following definition makes sens.    
\begin{defi} Let $\tilde K $ be a Cantor set of positive measure included in $[0,1]$.
Put \[\hat P(\tilde K):= \inf\{P(\hat K):\; \hat K\subset \tilde K \text{ compact subset of Lebesgue positive measure}\}\]
\end{defi}

\begin{ques}\label{quesopti}  If $\tilde K $ is a Cantor set of positive measure and containing a positive set of translations of a regular Cantor set of dimension $d$. Does $d+\hat P(\tilde K)\le 1$?
\end{ques}
We will answer negatively to this question in \textsection \ref{negarepo} thanks to a rather artificial example. On the other hand, for many stable examples, Proposition \ref{coro1} is close to be optimal.  

\subsection{A simple Example for which Proposition \ref{coro1} is stably close to be optimal}
We are going to prove the following result:
\begin{prop}\label{example}
For every $d\in (0,1)$, for every $p>1-d$, there exists a regular Cantor set $K$ of dimension $d$, a Cantor set of positive measure $\tilde K$ which satisfies Condition ($\mathcal C_p$) and such that for every $\lambda>0$, and for every $t$, the set  $\lambda K+ t$ is not included in $\tilde K$.
\end{prop}
\begin{rema} This example is stable by small $C^2$-perturbations of $K$.
\end{rema}
The proof of this Proposition uses the following.
\begin{theo}[\cite{Moreira}, Prop2.1, Theorem II-1.2] \label{gugu}Given $1>d,d^\prime>0$ such that $d+d^\prime>1$, there exist two regular Cantor sets $K$ and $K^\prime$ with Hausdorff dimensions respectively $d$ and $d^\prime$, that intersect stably. Besides, for any $\lambda>0$ and any real number $t$, $\lambda K+t$ is contained in a gap of $K^\prime$ or $K^\prime$ is contained in a gap of $\lambda K+t$, or $(\lambda K+t)\cap K^\prime\not= \emptyset$. Moreover, this property is $C^2-$open in $(K,K^\prime)$.
\end{theo}

\begin{proof}[Proof of Proposition \ref{example}] 
Let $0<d,p<1$ such that $d+p>1$. Let $K$ and $K^\prime$ be given by Theorem \ref{gugu} with $d$ and $d^\prime <p$. We can suppose that the convex hulls of $K$ and $\tilde K$ are $[0,1]$.  
We are going to construct a compact set $\tilde K$ of positive measure included in the following set:
\[X:= [0,1]\setminus \bigcup_{r\in \Q \cap (0,1)} (K^\prime+r)\]
The set $X$ has full measure but is not compact. Nevertheless $X$ cannot contain $\lambda\cdot K+t$ for every $\lambda>0$ and $t\in \R$. For the sake of contradiction, suppose this holds. Then $0< t\le 1-\lambda$. Let $r\in \Q \cap (t,t+\lambda)$. Then, since $\lambda < 1$, none of the Cantor sets $\lambda\cdot K+t, K^\prime+r$ can be contained in a gap of the other, and so, by the above theorem, they should have non-empty intersection, which contradicts the hypothesis that $\lambda\cdot K+t$ is contained in $X$. 

It remains to construct a Cantor set which satisfies condition $(\mathcal C_p)$ and which is included in $X$. 

Let $M\ge 1$ be large, and for every $r=a/b\in \Q$  put  $\eta_r:=  b^{-M}$. We can now define:
\[\hat  K:= [0,1]\setminus \bigcup_{r\in \Q \cap [0,1]} V(K^\prime+r,\eta_r),\]
where $V(K^\prime+r,\eta_r)$ is the neighborhood of $K^\prime+r$ made by an open cover of it by $C_{K^\prime}\eta_r^{-d^\prime}$ intervals  of length $\eta_r$, with $C_{K^\prime}$ the $d^\prime$-box fuzzy measure.
Hence the set $\hat K$ satisfies condition ($\mathcal C_p$)  if:
\[\sum_b b\times C_{K^\prime}\eta_r^{-d^\prime}  \eta_r^p<\infty.\]
As $\eta_r=b^{-M}$, the above sum converges if $1+M d^\prime-M p<0$, which can be obtained for $M$ large since $d^\prime<p $. 

The set $\hat K$ is compact, of positive Lebesgue measure (for $M$ large enough), and does not contain any interval.

 Let $\tilde K$ be the set of condensation points\footnote{Points which have a neighborhood whose intersection with $\hat K$ is uncountable}. Then $\tilde K$ is a Cantor set, satisfies also ($\mathcal C_p$) (see remark \ref{exprediff}), has the same positive measure as $\hat K$ (by Cantor-Bendixson theorem) and does not contains $\lambda K+t$ for any $\lambda>0$ and $t$.  
\end{proof}
\subsection{Counter examples to the optimality}
\subsubsection{A counterexample to the optimality of Theorem \ref{theo1}}\label{A paradigmatic example of non optimality} 
For $p>0$, let $N_p$ be a large integer defined afterward and  let $\tilde K_p$ be the Cantor set obtained by removing from $[0,1]$ the following intervals:
\[U_{iq} := \frac q{10^i} +(0, 10^{-[i/p]}),\quad i\ge N_p,\; q\in \{0,\dots , 10^i-1\}.\]
We remark that $\tilde K_p$ has positive measure whenever $p<1$ and $N_p$ is sufficiently large. Moreover $\tilde K_p$ satisfies condition $\mathcal C_{p'}$ for every $p'>p$. The intervals $(U_{iq})_{iq}$ are nested or disjoint, even if we remove those which are nested, $\tilde K_p$ does not satisfy Condition ($\mathcal C_{p'}$), for $p'\le p$.

Therefore, $P(\tilde K_p)=p$.

We remark that $U_{iq}$ is formed by points $u$ with the following decimal expression:
\[u= 0,u_1 u_2\cdots u_i\underbrace{0\cdots0}_{(j-i)\times}  u_{j+1}\dots u_n\cdots,\quad (u_l)_l\in\{0,\dots ,9\}\] 
such that $q= \sum_{l=1}^i u_l 10^{i-l}$, $j= [i/p]$,  and $u_{j+1}\cdots u_n\cdots$ any numbers in $\{0,..9\}$ (not all equal to zero nor almost all equal to $9$). 

Let $K$ be the regular Cantor set of $[0,1]$ formed by the points $k=0,0 k_1 k_2\cdots k_j\cdots$ such that $k_i$ are even number in $\{0,\dots, 8\}$. The Hausdorff dimension of such a Cantor set is $d= \ln(5)/\ln(10)$. The diameter of $K$ is less than $1/10$. 

From fact \ref{lefait}, the set $\mathbb X$ of parameters $t$ such that $t+K\subset \tilde K_p$ is included in $[0,  \diam\, \tilde K_p-\diam\, K]$ and satisfies:
\[[0,  \diam\, \tilde K_p-\diam\, K] \setminus \mathbb X= \bigcup_{k\in K, i\ge N_p} \bigcup_{0\le q\le 10^i} \big\{\frac{q}{10^i} -k\big\} +(0,10^{-j}),\]
It is an union of intervals of same length. Let us observe that if the measure of this set is smaller than $\diam\, \tilde K_p-\diam\, K$, then $t+K\subset \tilde K_p$ for a Lebesgue positive set of parameters $t$.

Let $K_i$ be the subset of points in $K$ formed by numbers $k$ such that $10^i k$ has fractional part $0$. Note that 
\[[0,  \diam\, \tilde K_p-\diam\, K] \setminus \mathbb X\subset  \bigcup_{i} \bigcup_{0\le q< 10^i,\, k\in K_j } \{q/10^i -k\} +[- 10^{-j},10^{-j}],\quad \text{with}\; j= [i/p].\]

Let us bound from above the cardinality $C_i$ of $\{q/10^i -k;\; 0\le q< 10^i,\, k\in K_j \}$.
Every $k\in K_j$ is of the form $k=k_1+10^{-i}k_2$ with $k_1\in K_i$ and $k_2\in K_{j-i}$.  Also 
\[\{q/10^i -k; 0\le q\le 10^i,\, k_1\in K_i\}\subset 10^{-i}\cdot  \big\{-10^i+1,\dots ,0,\dots , 10^i-1\}\]
which is of cardinality $2 \cdot 10^i-1$. On the other hand,  $K_{j-i}$ is of cardinality $10^{d(j-i)}$. Hence:
\[C_i\le   2 \cdot 10^i\cdot 10^{d(j-i)}\]
Consequently:
\[\leb([0,  \diam\, \tilde K_p-\diam\, K]\setminus \mathbb X)\le \sum_{i\ge N_p} C_i 2\cdot  10^{-j}\le \sum_{i\ge N_p}  2\cdot  10^i\cdot 10^{d(j-i)} 2\cdot  10^{-j}= 4 \sum_{i\ge N_p}   10^{(d-1)([i/p]-i)}. \]

  Hence for  $p<1$, there exists $N_p$ such that $\leb([0,  \diam\, \tilde K_p-\diam\, K]\setminus \mathbb X)$ has measure smaller than 
  $\leb([0,  \diam\, \tilde K_p-\diam\, K])\ge 0,9$. Hence $\mathbb X$ has positive measure, and $ K+t$ is contained in $\tilde K_p$ for a Lebesgue positive set of parameters $t$.

\begin{rema} This example shows that Theorem \ref{theo1} is not optimal since we can choose $p$ close to 1 so that $p+d>1$. 
\end{rema}
\begin{rema}[Generalization of the counterexample] 
For the sake of simplicity, we chose $K$ with decimal expression formed only by even numbers, but we could have chosen any other subset $J$ of $\{0,...,9\}$, and have the same result. Observe that the Hausdorff dimension of $K$ is then $\ln(card\, J)/\ln(10)$. Moreover, we could have chosen any other basis than the decimal one. This leads  to the following observation. For any $d<1$, there exists a regular Cantor set $K$ of dimension greater than $d$, such that for any $p<1$, there exists a Cantor set $\tilde K_p$ which does not satisfy ($\mathcal C_p$) but which contains $K$ for a positive set of translation parameters.
\end{rema}

Nevertheless, this counterexample to the optimality of Theorem  \ref{theo1} does not seem to be stable by perturbation, as shown by the example in \textsection \ref{optimal}.

\subsubsection{Negative answer to Question \ref{quesopti}}\label{negarepo}

Let us come back to the example of \textsection \ref{A paradigmatic example of non optimality}.

 Let $\hat K\subset \tilde K_p$ be of Lebesgue positive measure.  Suppose that $\hat K$ satisfies condition $(C_{p'})$, for $p'<p$. 

Let $[0,1]\setminus \hat K=: \sqcup_k U_k$. For every $i\ge N_p$ and $k$, let $n_i(k)$ be the number of intervals of the form $U_{iq}$, among $q\in\{0,\dots 10^i-1\}$, contained in $U_k$.

We remark that if $n_i(k)=1$, then $\leb(U_k)\ge 10^{-i/p}$, and if  $n_i(k)\ge 2$, then $\leb(U_k)\ge n_i(k) 10^{-i}/2$.

Let us suppose for the sake of contradiction that:
\[\sum_k \leb(U_k)^{p'}\le M<\infty\]
This implies that for every $i$:
\begin{equation}\label{cardinalite}Card\{k: n_i(k)=1\}\le M 10^{ip'/p}\end{equation}

Let $I$ be an interval of density for $\hat K$, this means that $\leb(\hat K\cap I)> 3\leb(I)/4$.
On the other hand, for $i$ large, it holds $\sum_{k: U_k\subset I} n_i(k)10^{-i} \ge 3\leb(I)/4$ (since $(U_k)_k$ are $10^{-i}$-dense), and so by (\ref{cardinalite}):
\begin{equation}\label{ineq2}\sum_{k: U_k\subset I, n_i(k)\ge 2} n_i(k) 10^{-i}\ge 3\leb(I)/4- M 10^{i(\frac {p'}p-1)} \end{equation}
As $p'<p$, for $i$ large the above sum is greater than $\leb(I)/2$. Hence 
\begin{equation}\label{ineq2bis}\sum_{k: U_k\subset I, n_i(k)\ge 2} \leb(U_k)\ge \sum_{k: U_k\subset I, n_i(k)\ge 2} n_i(k) 10^{-i}/2\ge \leb(I)/4\end{equation}
A contradiction with $\leb(I) - \sum_{k: U_k\subset I} \leb(U_k)= \leb(I\cap \hat K)> 3\leb(I)/4$.
%

\subsection{Natural example for which Proposition  \ref{coro1} is optimal}\label{optimal}

Let $K$ be a regular Cantor set of dimension $d$ (for instance we can continue with $K$ as above, made by numbers in $[0,1]$ whose decimal expression contains only even numbers. We recall that the Hausdorff dimension of $K$ is $d=\ln(5)/\ln(10)$).

For every $i$, let $u_{ik}\in [0,1], \; 0\le k\le 10^i-1$ be uniformly and independently randomly chosen. For ever $p\in(0,1)$, with $j= [i/p]$, we define:
\[\tilde K_p := [0,1]\setminus \left(\bigcup_{i\ge i_p} \bigcup_{0\le k\le 10^i-1} u_{ik} +[0,10^{-j}]\right),\]
where $i_p$ is sufficiently large so that $\tilde K_p$ has positive Lebesgue measure.

\begin{Claim}
If $p+d\ge 1$, the set $\{t:t+K\subset \tilde K_p\}$ has zero Lebesgue measure almost surely.
\end{Claim}
\begin{proof} 
For a fixed $t\in \mathbb R$ and $i\ge i_p$, we look the probability of the following event:
\[\{(t+K)\cap \{u_{ik}+[0,10^{-j}]\}\not= \varnothing \Leftrightarrow  t\in ([0,10^{-j}]-K)
+u_{ik}.\]
As for Proposition \ref{propo1}, and using the fact that the Hausdorff measure of dimension $d$ of $K$ is positive, there exists $c>0$ such that $\leb([0,10^{-j}]-K)>c (10^{-j})^{1-d}$.  The probability of the above event is greater than $c (10^{-j})^{1-d}$. As the variable $u_{ik}$ are independent, it comes that:
 \[Prob\{(t+K)\subset [0,1]\setminus \cup_k U_{ik}\}= \prod_k Prob\{(t+K)\subset [0,1]\setminus U_{ik}\}\le (1-c10^{-(1-d)[i/p]})^{10^i}\]
\[\le exp(-c 10^i 10^{-(1-d)[i/p]})\] 
 As $d+p\ge 1$, the above probability is uniformly less than $exp(-c/10)$. Considering this probability for infinitely many $i$, it holds that for every $t$ the probability that $t+K\subset \tilde K_p$ is 0. By Fubini's theorem, with full probability, the Lebesgue measure of the set $\{t: t+K\subset \tilde K_p\}$ is 0.
  \end{proof}

\begin{rema} 
As a consequence, for every given sequence $(\lambda_n)_n$ of positive real numbers with $\lim \lambda_n=0$, with full probability, the sets $\{t: t+\lambda_n K\subset \tilde K_p\}$ have zero Lebesgue measure for all positive integer $n$.
\end{rema}

\begin{rema} 
In the case $p+d>1$, it is possible to prove that, with full probability, the set $\{t: t+K\subset \tilde K_p\}$ is empty. Indeed, for each fixed $i$, we have  
\[\{(t+K)\cap \{u_{ik}+[10^{-j}/4,3\cdot10^{-j}/4]\}\not= \varnothing \Leftrightarrow  t\in ([10^{-j}/4,3\cdot10^{-j}/4]-K)
+u_{ik},\]
and the probability of the above event is greater than $c (10^{-j}/2)^{1-d}\ge c (10^{-j})^{1-d}/2$. So we have
 \[Prob\big\{(t+K)\subset [0,1]\setminus \bigcup_k \big(u_{ik}+[\frac{10^{-j}}4,\frac{3\cdot10^{-j}}4]\big)\big\}= \prod_k Prob\big\{(t+K)\subset [0,1]\setminus \big(u_{ik}+\big[\frac{10^{-j}}4,\frac{3\cdot10^{-j}}4\big]\big)\big\} \]
\[\le (1-c 10^{-(1-d)[i/p]} /2)^{10^i}\le\exp(-c 10^i 10^{-(1-d)[i/p]}/2)<\exp(-c\cdot10^{(d+p-1)i/p}/20).\] 
For each positive  integer $ r \le 2\cdot 10^j-1$, let $t_r=\frac{2r+1}{4\cdot 10^j}$. The probability that for one such integer $ r$, we have $(t_r+K)\subset [0,1]\setminus \cup_k (u_{ik}+[\frac{10^{-j}}4,\frac34 10^{-j}])$ is at most $2\cdot 10^j \exp(-c\cdot10^{(d+p-1)i/p}/20)\le 2\cdot 10^{i/p} \exp(-c\cdot10^{(d+p-1)i/p}/20)$, which tends (superexponentially fast) to $0$ when $i$ grows. So, with full probability, for $i$ large, and for each $r$ integer with $0\le r \le 2\cdot 10^j-1$, there is $k<10^i$ such that $(t_r+K)\cap (u_{ik}+[10^{-j}/4,3\cdot10^{-j}/4]) \not= \varnothing$, but if $x\in K$ and $(t_r+x)\in (u_{ik}+[\frac{10^{-j}}4,\frac 34 10^{-j}])$ then $([\frac r{2\cdot 10^j}, \frac {r+1}{2\cdot 10^j}]+x)=([t_r-\frac 1{4\cdot 10^j}, t_r+\frac 1{4\cdot 10^j}]+x)\subset u_{ik} +[0,10^{-j}]$, and so, for every $t \in [\frac r{2\cdot 10^j}, \frac {r+1}{2\cdot 10^j}]$, we have $(t+K)\cap [0,1]\setminus \cup_k U_{ik}\not= \varnothing$, and so $(t+K)\not \subset \tilde K_p$. So, for every $t\in [0,1]=\cup_{r=1}^{2\cdot 10^j-1}[\frac r{2\cdot 10^j}, \frac {r+1}{2\cdot 10^j}]$, $(t+K)\not \subset \tilde K_p$.
\end{rema}

\section{Toys  model for phase and parameter spaces of non-uniformly hyperbolic attractors}
\label{Explanation of the toy  model of the parameter selection for non-uniformly hyperbolic attractors}

\subsection{Toys  model for phase and parameter spaces of non-uniformly hyperbolic quadratic maps}
We recall that  if  a quadratic map $f(x)=x^2+a$  has an absolutely continuous invariant measure (acim) with Lyapunov exponent $\lambda>0$, then for  $C>0 $ small enough, $\lambda' \in (0,\lambda)$, the set $\tilde K$ of points $x$ such that:
\begin{equation}\tag{$\star$} 
 \|D_xf^n\| \ge  {C\lambda'^n},\quad \forall n\ge 0.
\end{equation}
is a Cantor set of positive Lebesgue measure, called Pesin set. The Collet-Eckmann property is that the critical value belongs to such a $\tilde K$ (for a certain  $C,\lambda'$).


The simplest example of non uniformly hyperbolic map is  the Chebichev map $x\mapsto  x^2-2$. It is semi conjugated to the doubling angle maps. 
 The first return map into $[-1,1]$ of the Chebichev map $x\mapsto  x^2-2$ is smoothly conjugated to the following map:
\[f\colon x\in [-1,1]\setminus \{0\} \mapsto \left\{\begin{array}{cl} 
2^{n+2} (x-2^{-n-1})-1,& \text{if }x\in (2^{-n-1}, 2^{-n}]\\
2^{n+2} (-x-2^{-n-1})-1,& \text{if }-x\in(2^{-n-1}, 2^{-n}]\end{array}\right.
\]
Put $I_{n+2}:= (2^{-n-1}, 2^{-n}]$ and $I_{-n-2}:= [-2^{-n}, -2^{-n-1} )$. Both intervals $int\, I_{\pm (n+2)}$ are sent diffeomorphically onto $(-1,1)$ by $f$. They are expanded by a factor $2^{n+2}$.  

Let $\mathcal D:= \cap_{k\ge 0}f^{-k}([-1,1])$ which is equal to $[-1,1]$ but a countable set.

Every $x\in \mathcal D$ has its image by $f^k$ into a certain $I_{x_k}$, with $x_k\in \mathbb Z\setminus \{-1,0,1\}$.

\paragraph{Geometry of Pesin sets for the Chebichev map}
We recall that a set of points which satisfies $(\star)$ with uniform $\lambda'$ and $C$ is called a \emph{Pesin set}.

Using the bijection map $x\in \mathcal D\mapsto (x_k)_{k\ge 1}$, we define the following family of Cantor sets:
\[\tilde K_{2}(s,N):= cl(\{x\in \mathcal D: |x_i|\le \max(N,s \sum_{j< i} |x_j|),\quad \forall i\}).\]

The conjugacy with the first return Chebichev map preserves $0$ and is smooth, hence the derivative of the Chebichev map at the points corresponding to those in $I_n$ is of the order of $2^{-n}$.

For every $s\in(0,1)$ and $N$, the set of points corresponding for the Chebichev map to $\tilde K_{2}(s,N)$ is included in a Pesin set. 
On the other hand, every Pesin set of the Chebichev map corresponds to a set included in $\tilde K_{2}(s,N)$, for some $s\in (0,1), N\ge 1$.


\begin{Claim} \label{exemchebichev}
The best lower bound of conditions ($\mathcal C_p$) satisfied by $\tilde K_{2}(s,N)$ is:
\[P(\tilde K_{2}(s,N))= \frac1{1+s}.\]
\end{Claim}
\begin{proof}
 Since $|f'|_{I_k} \equiv 2^{-|k|}$, $\cap_{j=0}^{i-1}f^{-j}(I_{x_j})$ is an interval $\tilde I$ of size $2^{1-S}$, where $S=\sum_{0 \le j <i}|x_j|$, and, for $i$ large, the condition $|x_i|> \max(N,s \sum_{j< i} |x_j|)$ is equivalent to $|x_i|>s \sum_{j< i} |x_j|=sS$, and so corresponds to a gap in the middle of $\tilde I$ and with proportion in $\tilde I$ equivalent to $2^{-s \sum_{j< i} |x_j|}=2^{-sS}$, so with size equivalent to $|\tilde I|^{1+s}$. These are (except for the small values of $i$) the connected components of the complement of $\tilde K_{2}(s,N)$.

For every $t>0$, for every $i\ge 1-\log_2 (t)$, any interval $\tilde I_i$ of the form $\tilde I_i=\cap_{j=0}^{i-1}f^{-j}(I_{x_j})$ has length less than $t$. Suppose that $|x_j|> \max(N,s \sum_{k< j} |x_k|)$ for every $j<i$ and put $S_i:= \sum_{j=0}^{i-1} |x_j|$. Note that $\tilde K_2 (s,N)$ is covered by such intervals $\tilde I_i$.  The sum of such $|\tilde I_j|$ is at least equal to $\leb(\tilde K_2(s,N))$.
Hence the series of their gaps at the power $1/(1+s)$ is equivalent to at least $\leb(\tilde K_2(s,N))$.

 We notice that the Lebesgue measure of $\tilde K_2(s,N)$ is positive. As this series counts only the gap smaller than $t$, the series of the power $1/(1+s)$ of the gap sizes diverges. Consequently, $P(\tilde K_2(s,N))\ge 1/(1+s)$.

On the other hand, the cardinality of such $\tilde I_j$ of size $2^{-k}$ among $j\ge 1$ is at most $2^{k+1}$ (such  $\tilde I_j$ are disjoint since they are all associated to the same return time $k$). Hence the series of the at the power $p$ is smaller than $\sum_k 2^{k+1} 2^{-(1+s)k p }$ which converges if $p>1/(1+s)$. Consequently, $P(\tilde K_2(s,N))\le 1/(1+s)$. \end{proof}

 Hence the union of the Pesin sets is equal to an increasing countable  union of Cantor sets whose $P$ approach 1/2.  One can easily show that the union of Pesin sets of a Chebichev map of higher degree is equal also to an increasing countable union of Cantor sets whose $P$ approach 1/2.  
\begin{ques} Given any Collet-Eckmann map, does the union of the Pesin sets is equal to an  increasing countable union of Cantor sets whose $P$ approach 1/2?\end{ques}
 
\paragraph{Toy model for 	a positive subset of Collet-Eckmann parameters}

From Jacobson' Theorem, the set of Collet-Eckmann parameters is of positive Lebesgue measure \cite{BC1, Yoc}. A way to prove it is to construct inductively such 'good' intervals $I_n$. The induction start with a bunch of such intervals $(I_n)_{|n|\le M}$ (given by hyperbolic continuity). 
 Then, we suppose that the first iterations of critical orbit went through the inductively constructed interval $I_n$, and this implies implies the existence of new intervals. This gives more and more possibilities for the post critical orbit. Every time that the critical orbit goes to an interval of last generation (ie $I_n$, $|n|>M$),  hyperbolicity is lost in the construction of the new intervals. To guaranty the Collet-Eckmann condition,  the critical point should not pass,  in mean,  by too much intervals $I_n$, for $|n|>M$. This gives rise to the notion of strong regularity \cite{Yoc}. This gives rise to the notion of strong regularity \cite{Yoc}. As far as the post critical orbit is hyperbolic, there are good distortion bounds for the transfer phase-parameter spaces, using for instance the puzzle-parapuzzle dictionary. 
%
%

 Using the above coding of $\mathcal D$,  for $M\ge 3$ and $\delta \in (0,1/2)$, let:
\[\tilde K_3(M,\delta):= \big\{x\in [-1,1] \; \mathrm{s.t.}\; \forall k\ge 0,\quad  \sum_{j\le k:\; |x_j|>M} |x_j|\le 
 \delta\sum_{j\le k} |x_j|\big\}\]
For $\delta= 2^{-\sqrt{M}}$, this set corresponds, in the dictionary puzzle-parapuzzle, to the set of strongly regular quadratic maps \cite{Yoc}. These quadratic maps satisfy the Collet-Eckman condition. 
\begin{Claim} For $M\ge 3$ it holds:
\[ 1-\delta\le P(\tilde K_3(M,\delta))\le 1-\delta+ \frac{\delta}{M} \log_2 \frac{e^2 M^2}\delta.\]
\end{Claim}
\begin{proof}
First we notice that $\tilde K_2(M, \delta/(1-\delta))$ contains $\tilde K_3(M,\delta)$, and so by Claim \ref{exemchebichev}:
\[P(\tilde K_3(M,\delta))\ge P(\tilde K_2(M, \delta/(1-\delta))=1-\delta.\] 

Let $\mathcal E(N,R,n,t)$ be the set: 
\[\big\{(a_i)_{i=1}^n\in (\Z\setminus \{-1,0,1\})^n :\; N=\sum_{j\le n} |a_j|,\; R=\sum_{j\le n:\; |a_j|>M} |a_j|,\; t=|\{j\le n:\; |a_j|>M\}|\big\}.\] 
Put $\mathcal E(N)=\cup_{R= [\delta N]+1,\;  n,\; t}\mathcal E(N,R,n,t)$. 
For every hole $H$ of  $\tilde K_3(M,\delta)$, there exist (exactly two) $(a_i)_{i=1}^{n}\in \mathcal E(N)$ such that 
$H$ is included the intervals $J_{(a_i)_{i=1}^{n-1}}:=\{x:\quad |x_1|=a_1, |x_2|=a_2, \cdots, |x_{n-1}|=a_{n-1}\}$, with $n$ maximal. Note that the size of $H$ is at most $2^{-N+3}$. 
%
%
 Hence:
\[\forall p\in(0,1),\; \sum_{N\ge 0} card(\mathcal E(N)) 2^{-pN}<\infty\Rightarrow P(\tilde K_3(M,\delta))\le p.\] 
Therefore the following implies that  $P(\tilde K_3(M,\delta))$ is at most $1-\delta +\frac \delta M \log_2(\frac{e^2 M^2}{\delta})$.

\begin{lemm} The cardinality of $\mathcal E(N)$ is at most $N^3 2^{N-\delta N}(e^2 M^2/\delta)^{(\delta N+1)/M}$.
\end{lemm} 
\end{proof}
%
%
\begin{proof}  
 For $R$, $N$, $t$ and $n$ fixed, the number of possible choices of $E= \{j\le n:\; |a_j|>M\}$ is ${n \choose t}$.
\begin{lemm} ${n \choose t}\le (eNM/2R)^{R/M}$.\end{lemm} 
 \begin{proof}  We remark that   $t \le R/M$. There are two cases:\begin{itemize}
\item  If $R/M < n/2$ then ${n \choose t}\le{n \choose R/M} \le  (enM/R)^{R/M}\le (eNM/2R)^{R/M}$.
\item If $R/M\ge n/2$ then ${n\choose t}\le 2^n\le (3e/2)^{n/2}\le (eM/2)^{R/M} \le (eNM/2R)^{R/M},\text{ since }R\le N$.\end{itemize}
\end{proof}
For given  $R$, $t$ and $E:= \{j\le n:\; |a_j|>M\}$, the number of possible choices of $(a_j)_{j\in E}$ is bounded by $2^t$ (to choose the sign of $a_j$) times the number of solutions of $x_1+x_2+\cdots+x_t=R$ (to choose the absolute value), so it is $2^t{R \choose t}\le 2^t{R \choose R/M}\le (2eM)^{R/M}$.

Therefore,  given  $N$, $R$, $n$ and $t$, the number of possible choices of $\{a_j, j\le n:\; |a_j|>M\}$ is at most $(e^2N M^2/R )^{R/M}$.

The number of choices of $(a_j, j\le n:\; |a_j|\le M)$ is bounded by $2^{N-R}$. Indeed an induction shows that the number $C_k$ of finite sequences of integers in $\mathbb Z\setminus \{-1,0,1\}$ whose absolute value sum is equal to $ k\ge 0 $ is less than $2^{k}$, since $C_{k+1}= 2+\sum_{i=2}^{k-2} 2 C_{k-i}$.  
  
So, the cardinality of $\mathcal E(N,R,n,t)$ is at most 
%
%
 \[(e^2 M^2 N/R)^{R/M} \cdot 2^{N-R}\le 2^{N-R}(e^2 M^2/\delta)^{R/M}\le 2^N(e^2 M^2 2^{-M}/\delta)^{R/M}=2^N(e^2 M^2 2^{-M}/\delta)^{(\delta N+1)/M}.\]
Given $N$ there are at most $N^3$ choices for $(t,R,n)$ 
 and so:
\[Card\; \mathcal E(N)\le N^3 2^N(e^2 M^2 2^{-M}/\delta)^{(\delta N+1)/M}
=  N^3 2^{N(1-\delta)}(e^2 M^2/\delta)^{(\delta N+1)/M}.\]\end{proof}
%
%

\begin{rema} The above estimates do not seems to be affected by the Collet-Eckmann map we look at.\end{rema}

\subsection{Toy models for phase and parameter spaces of non-uniformly hyperbolic Hénon maps}

In dimension 2, the  only known paradigmatic family of maps which has a non-uniformly hyperbolic attractor for a positive set of parameters (and this for generic perturbations) is the Hénon family. 
\[f_{ab}\colon(x,y)\mapsto(1-ax^2+y,-bx)\]
From Benedicks-Carleson Theorem \cite{BC1, BY, berhen}, for every fixed determinant $b$ (very) small, there exists a Lebesgue positive set of parameters $a\in \mathbb R$, such that $f_{ab}$ has  a non-uniformly hyperbolic attractor.
This means that  there exits an \emph{SRB probability} $\mu$ for $f_{ab}$: $\mu$ is invariant,  has a positive Lyapunov exponent, and the conditional measure of $\mu$ with respect to the unstable manifolds is absolutely continuous. 

The historical motivations were numerical experimentations by Hénon, wondering the existence of such a chaotic behavior for the parameter $a=1.4$ and $b=0.3$ of the Hénon family \cite{}. This problem is still open since the above works supposed $b$ much smaller than $0.3$. We will see at the end how the present work might help to solve this.

Let us give the flavor of the geometry of these attractors when $b$ is small. For such parameters $a$, there exists a uniformly hyperbolic horseshoe $H$, that  can  be fairly well approximated by the product of Cantor sets $K^u\times K^s$ where $K^u$ has Hausdorff dimension close to 1, $K^s$ has Hausdorff dimension close to 0, and the unstable direction is almost horizontal. The non-uniformly hyperbolic attractor is equal to closure of the unstable set of $H$.

Several constructions of the SRB measure \cite{BV,berhen} implies that restricted to any local unstable manifold $W^u_{loc}(z)\approxeq [-1,1]\times \{k\}$, $k\in K^s$, the set of points of $W^u_{loc}(z_0)$ which satisfies ($\star$) is a Cantor set $\tilde K$ of positive measure for any  $C>0 $ small enough, $\lambda'\in (0,\lambda)$. From the fact that the  Lyapunov exponent is large with respect to the determinant, every point $x$ of $\tilde K$ has a long local stable manifold $W^s_{loc}(x)$. Moreover the family of curves $(W^s_{loc}(x))_{x\in \tilde K}$ depends Lipschitz on $x$ in the meaning that there exists a Lipschitz homeomorphism which sends 
$\sqcup_{x\in \tilde K} W^s_{loc}(x)$ onto $\tilde K\times (-1,1)$. Moreover there exists $C'\in (0,C)$ such that $z\in \sqcup_{x\in \tilde K} W^s_{loc}(x)$ satisfies ($\star$) for $C'$ and $\lambda$.



The image by $f_a$ of each unstable manifold $W^u_{loc}(z)$ is a horizontal parabola. Similarly to the one dimensional case, the idea of the proof \cite{berhen} was (in particular) to make every image of $W^u_{loc} (H)$ by $f_a$ tangent to $\sqcup_{x\in \tilde K} W^s_{loc}(x)$. In other words, for every $z\in H$, there exists $x\in \tilde K$ such that 
$f(W^u_{loc} (z)) $ is tangent to $W^s_{loc} (x)$. In particular, there exists a diffeomorphism which sends $K^s$ into $\tilde K$.

\begin{figure}[h]
    \centering
        \includegraphics[width=10cm]{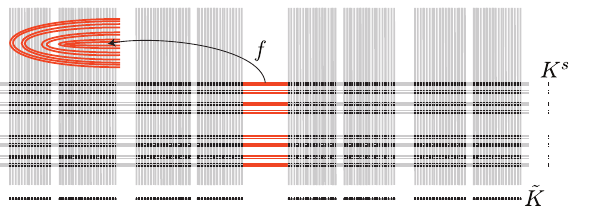}
    \caption{Toy model for parameter selection in the Hénon family.}
\label{EideltaS}
\end{figure}

When the parameter $a$ varies, the folding points have their $x$-coordinate which have nearly the same strictly monotone  dependence on $a$.

The parameter selection is handle in a same similar way as in dimension $1$, and roughly speaking, it is done by including in particular $K ^s$ into a similar Cantor set to $\tilde K$ for a Lebesgue positive set of translation parameters. 

\emph{It is a very rough description} :
Actually, the parameter selection is done during an induction which constructs an increasing sequence of unstable curves sets $([-1,1]\times K^s_n)_n$ and a decreasing sequence $(\tilde K_n)_n$ of Lebesgue positive subsets of $[-1,1]$, such that $K^s_0= K^s$ and $K^s_n$ is close to $K ^s$ and of Hausdorff dimension small, whereas $\tilde K=\cap_n \tilde K_n$.  
The construction of $\tilde K_n$ and $K^s$ depend of the $n$ first iterates $f^n(W^u_{loc}(H))$. Therefore, the geometries of $\tilde K_n$ and $K^s_n$ are constant only on smaller and smaller parameter intervals. These parameter intervals converge to single points.

\paragraph{Conclusion of the Toys model} 
From \cite{HZ}, the box dimension of the Hénon attractor for $a=1.4$ and $b=0.3$ is   $1.261\pm 0.003$. Hence the attractor should come from a uniformly hyperbolic horseshoe with stable transversal dimension close to $1$ and unstable transversal dimension close to $0.26$. On the other hand, the Cantor set $\tilde K({3,\delta})$ satisfies that $P(\tilde K({3,\delta})\approx 1-\delta$, and only $\delta\in (0,1/2)$ would be relevant to model the parameter space of non-uniformly hyperbolic maps. Hence, by Proposition \ref{coro1} and taking the notion of strong regularity with $\delta= 1/3>0.26$, this would be sufficient to nest the unstable transversal into the long stable foliation for a positive set of translation. Following the above toy model this would correspond to show Hénon conjecture\cite{He}.  

\bibliographystyle{amsalpha} \bibliography{biblio}

\end{document}